\documentclass{amsart}
\usepackage{amsmath,amssymb}
\usepackage{graphicx}
\newcommand{\R}{\mathbb R}
\newcommand{\Z}{\mathbb Z}

\newtheorem{thm}{Theorem}[section]
\newtheorem{lem}[thm]{Lemma}
\newtheorem{prop}[thm]{Proposition}
\newtheorem{conj}[thm]{Conjecture}
\newtheorem{cor}[thm]{Corollary}
\theoremstyle{definition}
\newtheorem{ex}[thm]{Example}
\theoremstyle{definition}

\theoremstyle{remark}
\newtheorem{rem}[thm]{Remark}
\newcommand{\lk}{\operatorname{lk}}
\def\sgn{\operatorname{sign}}
\def\sminus{\smallsetminus}
\def\<{\langle}
\def\>{\rangle}
\def\d{\partial}
\def\A{\mathcal{A}}

\def\D{\nabla}
\def\O{\Omega}
\def\S{\mathcal{S}}

\begin{document}

\title{Alexander-Conway invariants of tangles}
\author{Michael Polyak}
\address{Department of mathematics, Technion, Haifa 32000, Israel}
\email{polyak@math.technion.ac.il}

\begin{abstract}
We consider an algebra of (classical or virtual) tangles over an
ordered circuit operad and introduce Conway-type invariants of
tangles which respect this algebraic structure. The resulting
invariants contain both the coefficients of the Conway polynomial
and the Milnor's $\mu$-invariants of string links as partial cases.
The extension of the Conway polynomial to virtual tangles satisfies
the usual Conway skein relation and its coefficients are GPV finite
type invariants.
As a by-product, we also obtain a simple representation of the braid
group which gives the Conway polynomial as a certain twisted trace.
\end{abstract}

\thanks{Partially supported by the ISF grant 1343/10}
\subjclass[2010]{57M25; 57M27} \keywords{tangles, virtual links,
Gauss diagrams, Conway polynomial}

\maketitle

\section{Introduction}
The Alexander polynomial $\Delta_L(t)\in\Z[t,t^{-1}]$ of a link $L$ in $\R^3$
is one the most celebrated and well-studied link invariants. A number of
different definitions and approaches to $\Delta_L(t)$ are known (see e.g.
\cite{Ro, Tu}) and it is related to a variety of interesting objects and
constructions.
The Alexander polynomial has reappeared time and again in all major
developments in knot theory of the last decades: quantum invariants,
finite type invariants, and, lately, the theory of knot Floer homology
(see e.g. \cite{OS}).

In its original form, $\Delta_L(t)$ is only defined up to multiplication by
powers of $t$. Its close relative, the Conway polynomial
$\D(L)=\sum_n c_n(L)z^n\in\Z[z]$,
is free of this indeterminacy.
The Conway polynomial may be obtained from $\Delta_L(t)$ by a
substitution $z=t^{\frac12}-t^{-\frac12}$ and is completely
determined by its normalization $\D(O)=1$ on the unknot $O$ and
the Conway skein relation
$$\D(L_+)-\D(L_-)=\D(L_0),$$
\noindent
for any triple $L_+$, $L_-$ and $L_0$ of link diagrams, which look as shown
in Figure \ref{fig:Conway_triple} in a certain disk $B$
and coincide outside this disk.

\begin{figure}[htb]
\centerline{\includegraphics[width=3in]{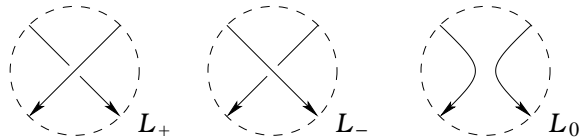}}
\caption{\label{fig:Conway_triple} Conway skein triple}
\end{figure}
The importance of the Conway skein relation was realized after the appearance
of the Jones polynomial, and led to the discovery of the HOMFLY polynomial
\cite{HOMFLY}.

The Alexander-Conway polynomial reappeared in the theory of quantum invariants
(where it turned out to be related to the quantum supergroup $U_q(gl(1|1))$, and also
to $U_q(sl(2))$ at roots of unity), and later in the theory of Vassiliev knot invariants,
since all coefficients of $\D(L)$ are finite type invariants (see \cite{BL}), with
$c_n(L)$ being an invariant of degree $n$.

Separate coefficients $c_n(L)$ of the Conway polynomial also attracted a lot of
attention, with coefficients of low degrees been extensively studied. E.g., for knots,
$c_2(L)$ is the Casson knot invariant. For 2-component links $c_1(L)$ is the linking
number of two components, and $c_3(L)$ is the Sato-Levine invariant.
For algebraically split 3-component links $c_2(L)$ is Milnor's triple linking number.
More generally, the first non-vanishing coefficient of $\D(L)$ for $n$-component
links is a certain combination of linking numbers \cite{Ho}; for algebraically split links
the first non-vanishing coefficient is a certain combination of Milnor triple linking
numbers \cite{Le, MV, Tr}. Coefficients $c_m(L)$ of a link $L$ are related to those
of a knot obtained by ``banding together'' components of $L$ (see \cite{Le}).

After the development of the theory of virtual knots (see \cite{Ka}), a number
of attempts was made to extend the definition of the Alexander or the Conway
polynomial to the virtual case using one of the original approaches for classical
links. In particular, J. ~Sawollek \cite{Sa} constructed a polynomial which,
however, vanishes on classical links.
One of the possible paths to pursue is that of the quantum invariants, or Fox
differential calculus. This allows one to generalize the Alexander polynomial to
(virtual or classical) tangles, but results in polynomials which are not invariant
under the first Reidemeister move and do not satisfy the skein relation
(e.g., a recent construction of Archibald and Bar-Natan \cite{ABN}).
We are unaware of any generalizations of the Alexander-Conway polynomial
to virtual links or tangles, which would coincide with the original polynomial
on classical links and satisfy the skein relation.

With this goal in mind, we follow the approach of Chmutov, Khouri and Rossi
\cite{CKR}, who use a tautological state sum model of Jaeger \cite{Ja} to
deduce Gauss diagram formulas for the coefficients $c_n(L)$ of $\D(L)$.
We extend and modify their construction to ordered tangles and provide a
direct proof of the invariance under the Reidemeister moves. This enables
us to extend the coefficients $c_n(L)$ to the virtual case as well.  The resulting
invariants are of finite type in the GPV sense \cite{GPV} (and thus change
under Kauffman's virtualization move \cite{Ka}).

We consider ordered tangles as an algebra over an ordered circuit operad.
Our Conway-type invariants of tangles respect this algebraic structure.
This allows us to ``break'' any complicated tangle into elementary fragments,
so all proofs and calculations may be done for elementary tangles.
The resulting invariants contain the coefficients of the Conway polynomial
of (long) links as a partial case.
For string links we use a certain shifted ordering to obtain Milnor's triple and
quadruple linking numbers.  We conjecture that all Milnor's homotopy
$\mu$-invariants of string links may be obtained in this way.
As a by-product, we also obtain a simple representation of the braid group,
which gives the Conway polynomial as a certain twisted trace. Due to the
Conway skein relation, this representation factors through the Hecke algebra.

The paper is organized in the following way.
In Section \ref{sec:tangles} we review classical and virtual ordered tangles.
Section \ref{sec:circuits} is dedicated to the ordered circuit operad and its
relation to ordered tangles.
In Section \ref{sec:invts} we introduce Conway-type invariants of tangles
and then discuss their properties in Section \ref{sec:properties}.

\section{Preliminaries}\label{sec:tangles}

\subsection{Classical and virtual tangles}
Let $B^3$ be the unit 3-dimensional ball in $\R^3$. An {\em
$n$-tangle} in $B^3$ is a collection $S$ of $n$ disjoint oriented
intervals and some number of circles, properly embedded in $B^3$ in
such a way, that the endpoints of each interval belong to the set
$X=\{x_k\}_{k=1}^{2n}$, where $x_k$ are some prescribed points on
the boundary of $B^3$. For example, one may choose points $x_k$ on
the great circle $z=0$, say, $x_k=(\exp(k\pi i/2n),0)$ for odd $k$
and $x_k=(\exp(-(k-1)\pi i/2n),0)$ for even $k$.

Tangles are considered up to an oriented isotopy in $B^3$, fixed on
the boundary.
We will call embedded intervals and circles the {\em strings} and the
{\em closed components} of a tangle, respectively. We will always
assume that the only singularities of the projection of a tangle to
the $xy$-plane are transversal double points. Such a projection,
enhanced by an indication of an over/underpass in each double point,
is called a {\em tangle diagram}. For technical reasons we will also
often fix a base point (distinct from the endpoints of strings) on
the boundary circle of the diagram. See Figure \ref{fig:tangles}a.

\begin{figure}[htb]
\centerline{\includegraphics[width=4.0in]{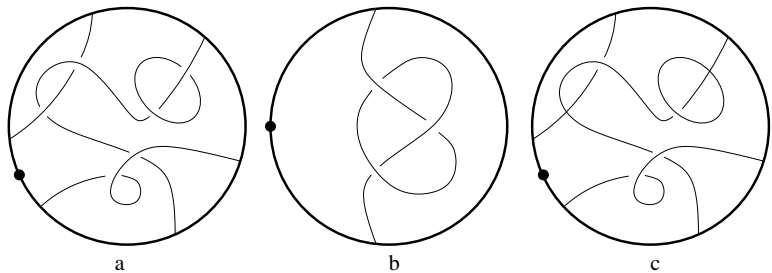}}
\caption{\label{fig:tangles} Classical and virtual tangle diagrams}
\end{figure}

Tangles generalize many objects, commonly considered in knot theory.
In particular, tangles which have no strings are usual links in $B$.
Tangles with one string and no closed components are {\em long
knots}, see Figure \ref{fig:tangles}b. Pure tangles without closed
components are called {\em string links}. Here a tangle is {\em
pure}, if the endpoints of $k$-th string, $k=1,2,\dots,n$ are
$x_{2k-1}$ and $x_{2k}$. A particular example of a pure tangle is
the unit tangle, with every pair $x_{2k-1}$ and $x_{2k}$ connected
by an interval.
Braids are tangles such that each tangle component intersects every
plane $y=c$, $c\in[0,1]$ in at most one point.

Virtual tangles present a useful generalization of tangles in the
framework of the virtual knot theory \cite{Ka}. In addition to usual
crossings of a tangle diagram, one considers a new -- virtual --
type of crossings. We will draw virtual crossings as double points
without any indication of the over- or underpass, see Figure
\ref{fig:tangles}c. Virtual tangle diagrams are considered up to
the classical Reidemeister moves (see Figure \ref{fig:reidem})
together with an additional set of virtual moves, shown in Figure
\ref{fig:virtual}.

\begin{figure}[htb]
\centerline{\includegraphics[width=3.8in]{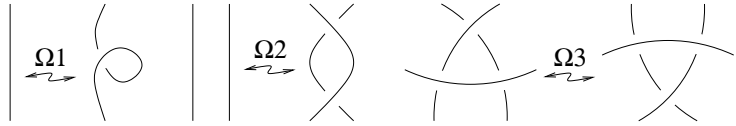}}
\caption{\label{fig:reidem} Classical Reidemeister moves}
\end{figure}
\begin{figure}[htb]
\centerline{\includegraphics[width=5in]{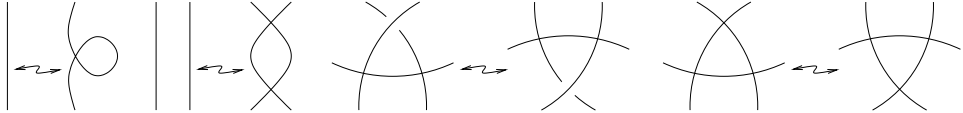}}
\caption{\label{fig:virtual} Virtual Reidemeister moves}
\end{figure}

Throughout the paper we will assume that all strings and closed
components are oriented. Also, further we will consider only tangles
with at least one string. To obtain invariants of closed classical
links, we will pick one of the components, cut it open, and consider
the resulting tangle with one string. We will then argue that the
result does not depend on the choice.
All our constructions will work both for classical and for virtual
tangles.

\subsection{Ordered tangles}\label{sub:order}
Let $T$ be a tangle with $n$ strings. Since we assume that all
strings are oriented, we can distinguish two endpoints of a string:
its {\em input} (or source), and its {\em output} (or target).
Denote by $\d^- T$ and $\d^+ T$ sets of inputs and outputs of all
strings of $T$, respectively. A tangle $T$ is {\em ordered}, if the
set $X=\cup_{i=1}^{2n}x_i=\d^- T\cup \d^+ T$ is numbered by a
collection $\{j_1<j_2<\dots<j_{2n}\}$ of integers, so that each
input is numbered by $j_{2k-1}$, and each output is numbered by
$j_{2k}$ for some $k=1,2\dots,n$. See Figure \ref{fig:ordering}.

\begin{figure}[htb]
\centerline{\includegraphics[width=4.0in]{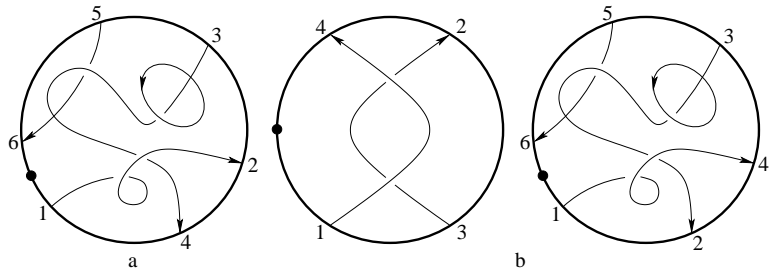}}
\caption{\label{fig:ordering} Coherent and non-coherent tangle
orderings}
\end{figure}

Orderings by different sets of integers which are related by a
monotone map will be considered equivalent. Note that closed
components do not appear in the definition of an ordering. Thus
ordered closed links are usual links with no additional data, and
1-string tangles (in particular, long knots) have a unique ordering.
We will call an ordering of a tangle $T$ {\em coherent}, if for
every $k=1,2,\dots,n$ a string with the input numbered by $j_{2k-1}$
has its output ordered by $j_{2k}$, see Figure \ref{fig:ordering}a.
For pure tangles (in particular string links) the standard ordering,
such that $x_k$ is labeled by $k$, is coherent.

\section{Circuit operad and tangles}\label{sec:circuits}

\subsection{Circuit diagrams}
A circuit operad \cite{ABN} is a modification of a planar tangles
operad \cite{J}, adjusted for virtual tangles instead of classical
tangles. It may be useful to look at Figure \ref{fig:circuit} for
some examples of circuit diagrams before reading the formal
definition below.

\begin{figure}[htb]
\centerline{\includegraphics[width=4.0in]{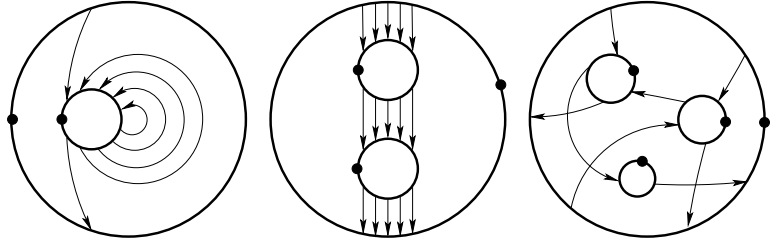}}
\caption{\label{fig:circuit} ``Trace'', ``composition'', and general
circuit diagrams}
\end{figure}

An {\em $n$-circuit diagram} $C$, $n>0$, is the unit disk $D_0$ in
$\R^2$ with a (possibly empty) collection of disjoint subdisks
$D_1,D_2,\dots,D_k$ in the interior of $D_0$. The boundary $\d D_0$
of the disk $D_0$ is called the {\em output} of $C$, and the union
of boundaries $\d D_i$ of $D_i$, $i=1,2,\dots,k$ is the {\em input}
of $C$. Each disk $D_i$ has an even number $2n_i$ of distinct marked
points on its boundary (with $n=n_0$). Each boundary circle $\d
D_i$, $i=0,1,\dots,k$ is based, i.e., equipped with a base point
distinct from the marked points; we will denote it by $*_i$. The set
of all marked points is equipped with a matching, i.e., is split in
$N=n_0+n_1+\dots+n_k$ disjoint pairs. A pair is called {\em
external}, if at least one of its points lie on the output $D_0$,
and is {\em internal} otherwise. It is convenient to think about
each pair as a simple path in the complement $D\sminus\cup_{i=1}^k
D_k$ of the internal disks, connecting the corresponding marked
points. These connections represent only the matching of marked
points, but not the actual path. Actual paths, as well as their
intersections, are irrelevant, so we will treat these intersections
as virtual crossings. Circuit diagrams are considered up to
orientation-preserving diffeomorphisms. If paths in a circuit
diagram may be realized without intersections in
$D\sminus\cup_{i=1}^k D_k$, we recover planar tangle diagrams of
\cite{J}.

Two circuit diagrams $C$ and $C'$ with the appropriate number of
boundary points may be composed into a new circuit diagram $C\circ_i
C'$ as follows. Isotope $C'$ so that the output $\d D'_0$ of $C'$,
together with the set of marked points and the base point, coincides
with the input $D_i$ of $C$. Glue $C'$ into the internal disk $D_i$
of $C$ (smoothing near the marked points so that the paths of $C$
and $C'$ meet smoothly) and remove the common boundary. See Figure
\ref{fig:compose}. This composition defines a structure of a colored
operad on circuit diagrams.

\begin{figure}[htb]

\centerline{\includegraphics[width=4.1in]{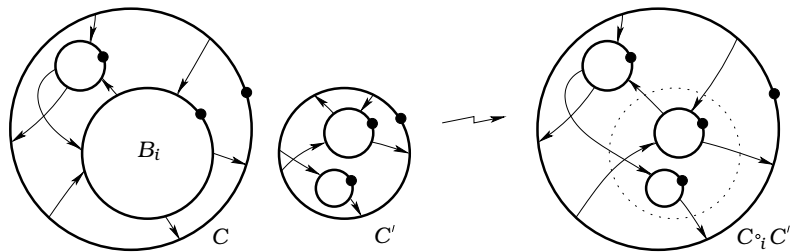}}
\caption{\label{fig:compose} Composing circuits}
\end{figure}

Further we will consider ordered oriented circuits. A circuit
diagram is {\em oriented}, if all paths connecting pairs of marked
points are oriented. In other words, each matched pair $e$ of marked
points is ordered: $(s(e),t(e))$ -- the source and the target of
$e$, respectively. The number of inputs and outputs on each circle
$\d D_i$ are required to be equal (and thus equal to $n_i$). See
Figures \ref{fig:circuit}, \ref{fig:compose}. Compositions of
oriented tangles should respect orientations of the paths.

\subsection{Ordered circuit diagrams}
An oriented circuit is {\em ordered}, if the set of paths is
ordered, i.e., the set of all pairs of matched marked points is
ordered: $(s_1,t_1),\dots,(s_N,t_N)$. We assume that two following
conditions hold. Firstly, points $s_1$ and $t_N$ should be on the
output $\d D_0$ of $C$. Secondly, both points $t_i$, $s_{i+1}$,
$i=1,2,\dots,N-1$ should be on the boundary of the same disk of $C$.
See the rightmost picture of Figure \ref{fig:order}.

An ordering of all external paths in $C$ is called an {\em external
ordering}, and an ordering of all internal paths in $C$ is called an
{\em internal ordering} of $C$. Given a (complete) ordering of $C$,
its restriction to external and internal paths defines an external
and an internal orderings of $C$, respectively. Vice versa, given both
an internal and an external ordering of $C$, one may construct
different complete orderings of $C$ as shuffles of these two partial
orderings, see Figure \ref{fig:order}.

\begin{figure}[htb]
\centerline{\includegraphics[width=4.0in]{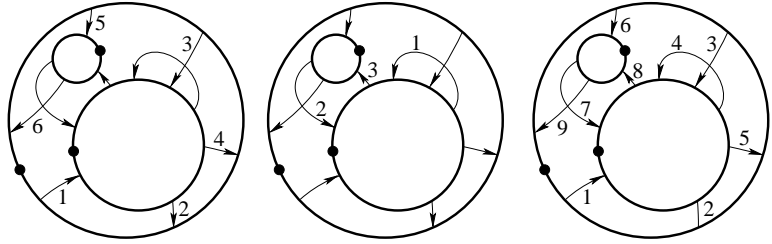}}
\caption{\label{fig:order} External, internal, and complete
orderings of a circuit}
\end{figure}

To define a composition $C\circ_i C'$ of ordered circuits $C$ and
$C'$, we require the following compatibility of orderings. A {\em
circuit} of length $l$ in a circuit diagram $C$ is a sequence
$(s_j,t_j),\dots,(s_{j+l},t_{j+l})$ of paths, such that both
endpoints $s_j$ and $t_{j+l}$ lie on the output $\d D_0$, while
the rest of the endpoints in this sequence lie on the inputs
$\cup_{i=1}^k\d D_i$. For example, an ordered circuit diagram in
Figure \ref{fig:order} contains three circuits of lengths $2$,
$3$, and $4$, respectively. If a sequential pair $t_i$ and $s_{i+1}$
of marked points on $\d D_i$ are identified with a pair $s'_{j}$ and
$t'_{j+l}$ of marked points on $\d D'_0$, then we require that
$s'_{j}$ and $t'_{j+l}$ are endpoints of a circuit in $C'$. The
ordering on $C\circ C'$ is then induced from orderings on $C$ and
$C'$.


\subsection{Inserting tangles in circuit diagrams}
\label{sub:orderCT}
Given an $n$-circuit diagram $C$ with $k$ inputs and a $k$-tuple of
tangles $T_1,\dots,T_k$, we may create a new tangle
$C(T_1,\dots,T_k)$ if the data on the boundaries match. Isotope an
$n_i$ tangle $T_i$ to such a position, that its boundary circle
coincides with the input circle $\d D_i$ of $C$ and the endpoints of
strings in $T_i$ coincide with the marked points on $\d D_i$ (if $C$
and $T_i$ are oriented, we also require a match of orientations).
Glue $T_i$ into the internal disk $D_i$ of $C$ along the boundary,
removing common boundary circles and thinking about paths in the
complement $D\sminus\cup_{i=1}^k D_k$ as a virtual tangle diagram.
Doing this for all disks $D_i$, we obtain a new virtual $n_0$-tangle
$C(T_1,\dots,T_k)$ in the disk $D_0$. See Figure
\ref{fig:composition}.

\begin{figure}[htb]
\centerline{\includegraphics[width=4.2in]{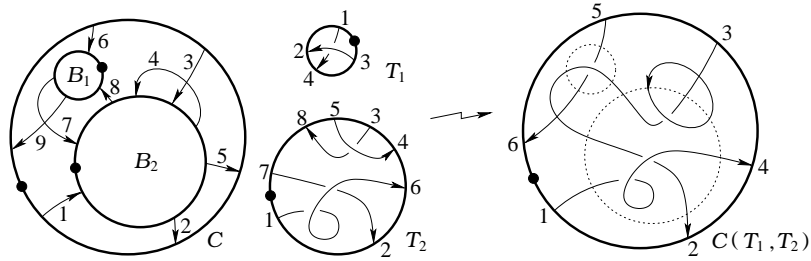}}
\caption{\label{fig:composition} Inserting tangles in a circuit}
\end{figure}

Note that in order to define an ordering of the resulting tangle
$C(T_1,\dots,T_k)$, we need to fix only an external orientation on
$C$; $T_i$'s need not to be ordered. Moreover, if $C$ has a
(complete) ordering, it induces an ordering on each tangle $T_i$.
Indeed, an ordering of $C$ defines a numbering
$i_1,o_1,\dots,i_N,o_N$ of the set of all marked points, which
induces an ordering of marked points on the boundary circle $\d
D_i$, and thus an ordering of $T_i$. If all $T_i$'s and $C$ are
ordered, for the composition to be defined we require that this
induced ordering should coincide with that of $T_i$. See Figure
\ref{fig:composition}.

This operation gives to virtual tangles a structure of an algebra
over the circuit operad, similar to the usual case of classical
tangles as an algebra over the planar operad.

\section{Invariants of ordered tangles}\label{sec:invts}

\subsection{States of tangle diagrams}\label{sub:states}
Let $D$ be an ordered tangle diagram. An {\em $n$-state} of $D$ is a
collection of $n$ crossings of $D$. A state $S$ of $D$ defines a new
tangle diagram $D(S)$, obtained from $D$ by smoothing all crossings
of $S$ respecting the orientation, see Figure \ref{fig:smooth}a. The
smoothed diagram inherits an ordering from $D$. We will say that the
state $S$ is {\em coherent}, if $D(S)$ contains no closed components
and the ordering of $D(S)$ is coherent (see Section
\ref{sub:order}), i.e. if both ends of each string numbered by
$j_{2i-1}$ and $j_{2i}$ for some $i$. Suppose that $S$ is coherent.
As we follow $D(S)$ along the first string of $D(S)$ (starting from
its input and ending in its output), then continue to the second
string of $D(S)$ in the same fashion, etc., we pass a neighborhood
of each smoothed crossing $s\in S$ twice. A (coherent) state $S$ is
{\em descending}, if we enter this neighborhood first time on the
(former) overpass of $D$, and the second -- on the underpass. See
Figure \ref{fig:smooth}b. The sign $\sgn(S)$ of $S$ is defined as
the product of signs (local writhe numbers) of all crossings in $S$.

\begin{figure}[htb]
\centerline{\includegraphics[width=4.0in]{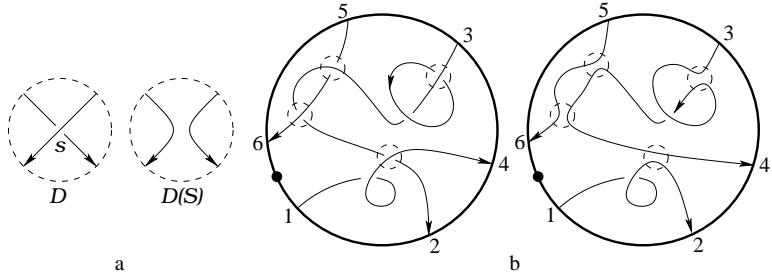}}
\caption{\label{fig:smooth} Smoothing crossings which give a
descending state}
\end{figure}

\begin{rem}
Property of coherency of a state depends on the ordering of a
diagram $D$. Given a state $S$ such that $D(S)$ contains no closed
components, there are $n!$ orderings of $D$ (which differ by
reorderings of the $n$ strings of $D(S)$) for which $S$ is coherent.
To estimate the number of orderings for which such a state $S$ is
descending, construct a graph with vertices corresponding to
components (both open and closed) of $D$, with two vertices
connected by an edge if $S$ contains a crossing between two
corresponding components of $D$. If the number of connected
components of this graph is $c$, there are at most $c!$ orderings of
$D$ for which $S$ may be descending. In particular, if the graph is
connected, there is at most one ordering for which $S$ may be
descending.
\end{rem}

\subsection{Conway-type invariants of ordered tangles}
\label{sub:conway}
Denote by $\S_n(D)$ the set of all descending $n$-states of a
diagram $D$. Define $c_n(D)\in\Z$ and $\D(D)\in\Z[z]$ by
$$c_n(D)=\sum_{S\in\S_n(D)}\sgn(S)\,,\quad \D(D)=\sum_{n=0}^\infty c_n(D) z^n$$
In particular, $c_0(D)=1$ if $D$ has no closed components and the
ordering of $D$ is coherent (indeed, in this case the empty state
$S=\emptyset$ is descending), and $c_0(D)=0$ otherwise.

\begin{ex}\label{ex:unlink}
For a trivial tangle diagram $D_0$ which consists of one straight
string, the only descending state is trivial, thus $\D(D_0)=1$.
Also, an addition of a kink to a string of a tangle does not change
the value of $\D$ (since the new crossing cannot enter a coherent
state). In particular, for a diagram $D'_0$ obtained from $D_0$ by
an addition of a small kink, we have $\D(D'_0)=1$. If a tangle
diagram $D_{split}$ is split -- i.e., it can be subdivided into two
non-empty non-intersecting parts -- then there are no coherent (and
thus no descending) states, thus $\D(D_{split})=0$.
\end{ex}

\begin{ex}\label{ex:O2}
Let $D_a$, $D_b$, $D_c$ and $D_d$ be the diagrams shown in Figure
\ref{fig:O2}. The only descending states for $D_a$ and $D_b$ are
trivial, so $\D(D_a)=\D(D_b)=1$. For $D_c$ there are no descending
states, so $\D(D_c)=0$. For $D_d$ there are two descending 1-states,
but their signs are opposite, so $\D(D_d)=z-z=0$.
\end{ex}

\begin{figure}[htb]
\centerline{\includegraphics[width=4.8in]{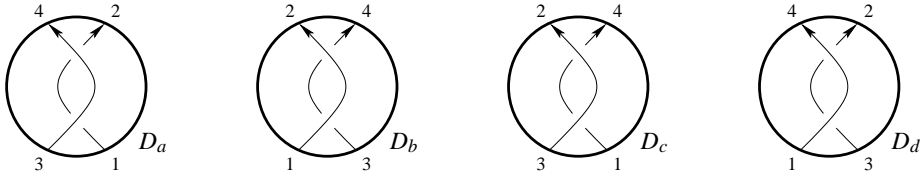}}
\caption{\label{fig:O2} Tangle diagrams with two strings}
\end{figure}

\begin{ex}
For a diagram $D_1$ of a long Hopf link in Figure \ref{fig:example}a
there is only one descending state $\{2\}\in\S_1(D_1)$. Thus
$\D(D_1)=\pm z$ (depending on the orientation of the closed
component). For the diagram $D_2$ of a long trefoil in Figure
\ref{fig:example}b there is only one non-trivial descending state
$\{2,3\}\in\S_2(D_2)$, so $\D(D)=1+z^2$.  For the diagram $D_3$ of a
long iterated Hopf link in Figure \ref{fig:example}c, there are
three non-trivial descending states: $\{2\},\{4\}\in \S_1(D)$, and
$\{2,3,4\}\in \S_3(D)$. Thus $\D(D_3)=\pm (2z+z^3)$ (depending on
the orientation of the closed component).
\end{ex}

\begin{figure}[htb]
\centerline{\includegraphics[width=4.2in]{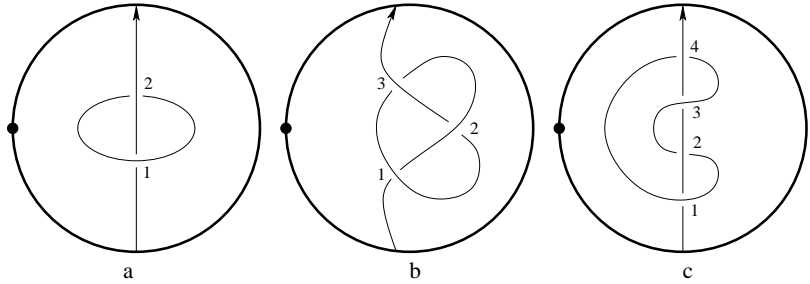}}
\caption{\label{fig:example} More tangle diagrams}
\end{figure}

Let $C$ be on oriented circuit diagram with an external ordering.
Let $T_1$,\dots,$T_k$ be oriented tangles such that the composition
$C(T_1,\dots,T_k)$ is defined. Note that the external ordering of
$C$ makes $C(T_1,\dots,T_k)$ into an ordered tangle. Now, suppose
that a (complete) ordering $or$ of $C$ extends the given external
ordering. As discussed in Section \ref{sub:orderCT}, $or$ induces an
ordering on each tangle $T_i$; denote the resulting ordered tangle
by $T_i^{or}$. Directly from the definition of $c_n$ we conclude
that $\D$ behaves multiplicatively under tangle compositions:
\begin{lem}\label{lem:prod}
We have
$$\D(C(T_1,\dots,T_k))=\sum_{or}\prod_{k=1}^k\D(T_i^{or})\, ,$$
where the summation is over all orderings $or$ of $C$, extending the
initial external ordering.
\end{lem}

\begin{thm}\label{thm:main}
Let $D$ be a diagram of an ordered (classical or virtual) tangle
$T$. Then $\D(T)=\D(D)$ defines an invariant of ordered tangles.
\end{thm}

\begin{proof}
Let us prove the claim by checking the invariance of $c_n(D)$ under
the Reidemeister moves $\O1-\O3$. The results of \cite{Po2} imply
that all oriented Reidemeister moves are generated by 4 moves shown
in Figure \ref{fig:reidem_tangles} below.
\begin{figure}[htb]
\centerline{\includegraphics[width=4.6in]{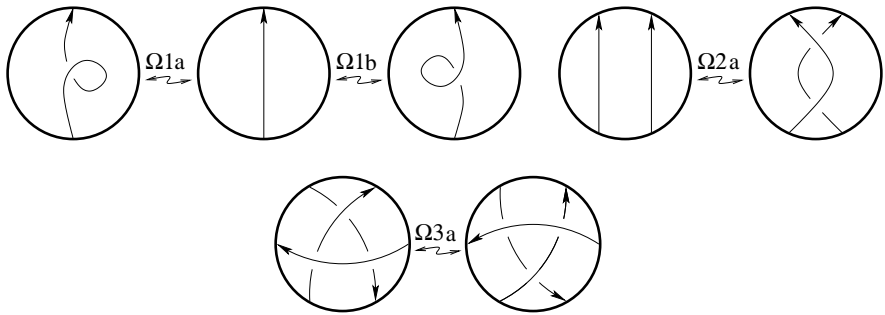}}
\caption{\label{fig:reidem_tangles} Elementary tangles related by Reidemeister moves}
\end{figure}

Due to Lemma \ref{lem:prod}, it suffices to verify the claim only on
pairs of elementary tangle diagrams of Figure \ref{fig:reidem_tangles}
for different orderings. Indeed, a pair of ordered tangles which are
related by an oriented Reidemeister move in a certain disk $B$ can
be presented as compositions $C(T_1,T_2)$ and $C(T'_1,T_2)$ with
$T_1$ and $T'_1$ being elementary tangle diagrams inside the disk
$B$, and $T_2$ being a tangle in a disk surrounding all other classical
crossings, as illustrated in Figure \ref{fig:compReidem}.

\begin{figure}[htb]
\centerline{\includegraphics[width=4.2in]{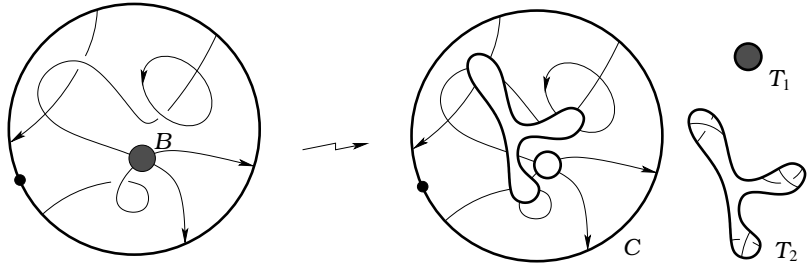}}
\caption{\label{fig:compReidem} Decomposing tangles into pieces}
\end{figure}

If a pair of ordered diagrams differ by $\O1$, the equality
$\D(T_1)=\D(T'_1)$ follows from Example \ref{ex:unlink}. A similar
equality for $\O2$ follows from Example \ref{ex:O2}.

Finally, if $T$ and $T'$ are two ordered tangle diagrams related
by $\O3$, there is an obvious bijective correspondence between
all descending 0- and 1-states of $T$ and $T'$; also, there are no
descending 3-states. This implies the required equality for every
ordering for which there are no descending 2-states of $T$ and
$T'$. But if for a certain ordering such states exist, they appear in
pairs with opposite signs, and thus cancel out, see Figure \ref{fig:O3}.

\begin{figure}[htb]
\centerline{\includegraphics[width=4.2in]{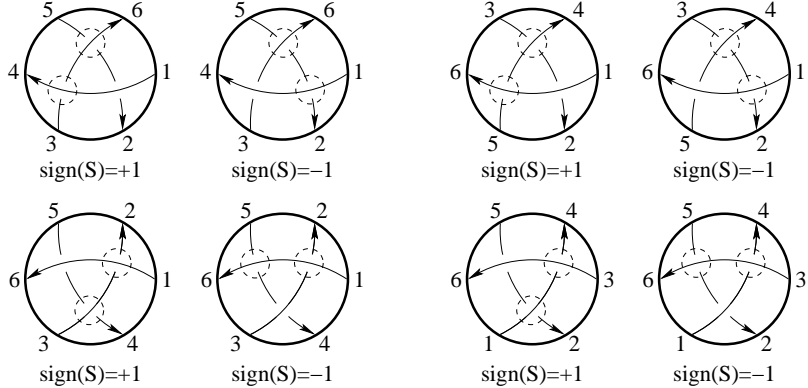}}
\caption{\label{fig:O3} Canceling pairs of descending states}
\end{figure}

\end{proof}

\section{Properties of the invariants}\label{sec:properties}

\subsection{Skein relation}
Our choice of notation $\D$ for the invariant of Theorem \ref{thm:main}
is explained in the following proposition:
\begin{prop}\label{prop:skein}
Let $T_+$, $T_-$ and $T_0$ be ordered tangles which look as shown in
Figure \ref{fig:skein}a inside a disk and coincide (including
orderings) outside this disk. Then the Conway skein relation holds:
$$\D(T_+)-\D(T_-)=z\D(T_0)$$
\end{prop}

\begin{figure}[htb]
\centerline{\includegraphics[width=4.0in]{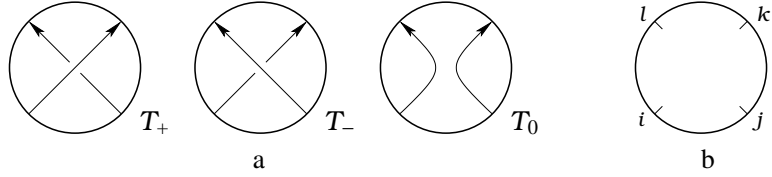}}
\caption{\label{fig:skein} Skein relation for tangles}
\end{figure}

\begin{proof}
Due to the multiplicativity property of Lemma \ref{lem:prod}, it
suffices to establish the equality for the three tangles of Figure
\ref{fig:skein}a equipped with one of the four possible orderings
shown in Figure \ref{fig:skein}b, with $(i,j,k,l)$ being $(1,3,2,4)$,
$(1,3,4,2)$, $(3,1,2,4)$, or $(3,1,4,2)$. The corresponding values
of $\D$ are shown in the table below.
%
$$\begin{array}{|c||c|c|c|c|}
(ijkl) & (1324) & (1342) & (3124) & (3142) \\ \hline
\hline \makebox(0,12){} \D(T_+) & 1 & z & 0 & 1 \\
\hline \makebox(0,12){} \D(T_-) & 1 & 0 & -z & 1 \\
\hline \makebox(0,12){} \D(T_0) &  0 & 1 & 1 &  0
\end{array}
$$
\end{proof}

The above proposition, together with the normalization of Example
\ref{ex:unlink} on split tangles and on the trivial 1-string tangle,
imply the following corollary:
\begin{cor}[c.f. \cite{CKR}]\label{cor:Conway}
Let $T$ be a classical tangle with one string\footnote{Recall that
for a tangle with one string there is only one possible ordering.}
and $m-1$ closed components. Denote by $\bar{T}$ a closed
$m$-component link, obtained from $T$ by the braid-type closure.
Then $\D(T)$ is the Conway polynomial of $\bar{T}$.
\end{cor}

\subsection{Relation to Milnor's $\mu$-invariants}
For $n>1$ invariants $c_n(T)$ include well-known Milnor's $\mu$-invariants
of link-homotopy \cite{Mi} and their generalizations. In particular, let $T$
be a tangle with $n$ ordered strings and no closed components.
Denote by $T^{\sigma}$ the tangle $T$ equipped with an ordering
which is obtained from the standard one by a cyclic permutation
$\sigma=(246\dots 2n)$, i.e. such that the source and the target of the
$k$-th string are numbered by $j_{2k-1}$ and $j_{2k+2}$, respectively,
for $k=1,\dots,n-1$, and the source and the target of the last string are
numbered by $j_{2n-1}$ and $j_2$. See Figure \ref{fig:mu}.
Denote $\nu_{12\dots n}(T)=c_{n-1}(T^{\sigma})$.

\begin{figure}[htb]
\centerline{\includegraphics[width=4.4in]{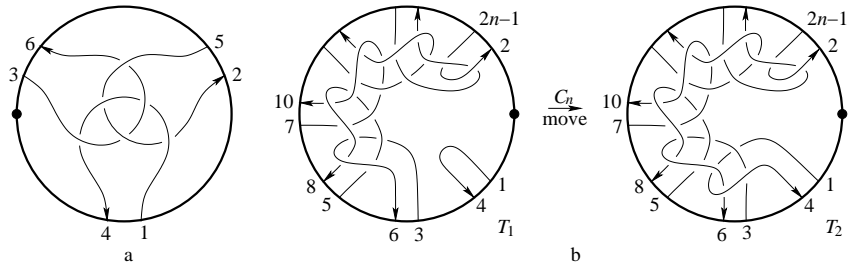}}
\caption{\label{fig:mu} Tangles with a shifted ordering}
\end{figure}

The invariant $\nu_{12\dots n}$ detects Goussarov-Habiro's $C_n$ move
\cite{Gu, Ha}: an easy calculation shows that for tangles $T_1$ and $T_2$
depicted in Figure \ref{fig:mu}b we have $\nu_{12\dots n}(T_1)=0$,
$\nu_{12\dots n}(T_2)=1$.

\begin{conj}
Modulo lower degree invariants, $\nu_{12\dots n}$ coincide with Milnor's
link homotopy $\mu$-invariant\footnote{or rather, its generalization
of \cite{KP} to (classical or virtual) tangles} $\mu_{2\dots n,1}$. In
particular, if a closure $\hat{T}$ is a Brunnian link, so that $n$-th
degree $\bar{\mu}$-invariants of $\widehat{T}$ are well-defined, we have
$\nu_{12\dots n}(T)=\bar{\mu}_{12\dots n}(\widehat{T})$.
\end{conj}

Currently, the general structure of the lower degrees correction terms
remains unclear.
Here are some explicit formulas for low degree invariants, which directly
follow from the results of \cite{KP, Po1} (after a straightforward translation
of formulas for $\nu_{123}(T)$ and $\nu_{1234}(T)$ to the language of
Gauss diagrams):

\begin{prop}
Let $T$ be a (possibly virtual) string link. Then
$$\nu_{12}(T)=\mu_{2,1}(T)=\lk_{21}(T)\ ,\quad
\nu_{123}(T)=\mu_{23,1}(T)+\lk_{13}(T)\lk_{32}(T)$$
\begin{multline*}\nu_{1234}(T)=\mu_{234,1}(T)+\mu_{34,1}(T)\lk_{32}(T)+
\mu_{24,1}(T)\lk_{43}(T)+\mu_{23,4}(T)\lk_{14}(T)+\\
\lk_{14}(T)\lk{43}(T)\lk_{32}(T)
\end{multline*}
\end{prop}

Cases when $k\ge n$ are no less interesting. Here are some results
for small values of $n$ and $k$.
Denote by $V_2(T)$ an invariant of 2-component string links  introduced in
 \cite{Me}, and let $v_2$ be the Casson knot invariant, i.e. the second
 coefficient of the Conway polynomial.

\begin{prop}
For a $2$-string link $T$ with a standard ordering we have $c_2(T)=V_2(T)$.
This invariant is a splitting of the Casson knot invariant $v_2$, in the following
sense. For a long knot $K$, denote by $K^2=K\cup K'$ a $2$-string link
obtained from $K$ by a ``reverse doubling'', i.e. by taking a pushed-off copy
$K'$ of $K$ along a zero framing, and then reversing its orientation.
Then $v_2(K)=-\frac12 V_2(K^2)$.
\end{prop}
\begin{proof}
The equality $c_2(T)=V_2(T)$ follows from equation 3.5 of \cite{Me}.
Moreover, equation 2.2 of \cite{Me} implies that for any 2-string link
$T$ with strings $T_1$ and $T_2$ one has
$V_2(T)= v_2(\bar{T}_{pl})-v_2(T_1)-v_2(T_2)$,
where $\bar{T}_{pl}$ is a plat closure of $T$. Apply this equality to
$T=K^2$ and note that the plat closure $\bar{K^2}_{pl}$ of $K^2$
is the unknot, thus $V_2(K^2)=-v_2(K)-v_2(K')$. Also, $v_2$ is
preserved under an orientation reversal, so $V_2(K^2)=-2v_2(K)$.
\end{proof}

\subsection{GPV finite type invariants}
Finally, let us show that all invariants $c_n(T)$ are finite type invariants
of virtual tangles in the sense of GPV \cite{GPV}. Recall, that for virtual
links there are two different theories of finite type invariants.
Kauffman's theory of \cite{Ka} is based on crossing changes similarly to
the case of classical knots.
GPV theory introduced in \cite{GPV} is more restrictive in a sense that
every GPV invariant is also a Kauffman invariant, but not vice versa.
In GPV theory instead of crossing changes one uses a new operation,
special for the virtual knot theory: crossing virtualization. Namely, given
a real crossing in a diagram of a virtual link, we can convert it into a virtual
crossing (resulting in a new virtual diagram).

Let $D$ be a virtual tangle diagram with $n$ marked and ordered
real crossings $d_i$, $i=1,\dots,n$. Given an $n$-tuple
$I=\{i_1,\dots, i_n\}\in\{0,1\}^n$ of zeros and ones, denote by $D_I$ a
virtual tangle diagram obtained from $D$ by virtualization of all crossings
$d_k$ with $i_k=1$. Also, let $|I|=\sum_{k=1}^n i_k$. An invariant
$v$ of virtual tangles is called a GPV finite type invariant of degree less than
$n$, if an alternating sum of its values on all diagrams $D_I$ vanishes:
$$\sum_{I\in\{0,1\}^n}(-1)^{|I|}v(D_I)=0$$
for all diagrams $D$ and choices of crossings $d_i$.
If $v$ is of degree less than $n+1$, but not less than $n$, we say that $v$
is a GPV invariant of degree $n$.
A restriction to classical tangles of a GPV invariant of degree $n$ is a
Vassiliev invariant of degree less than $n$.
\begin{rem}
An open (an highly non-trivial) conjecture is whether for long knots the
opposite is true, i.e., whether any Vassiliev invariant of classical long
knots can be extended to a GPV invariant. A positive solution would
imply that Vassiliev invariants classify knots.
\end{rem}

\begin{prop}
For any $n\in{\mathbb N}$, the invariant $c_n(T)$ is a GPV invariant
of degree $n$.
\end{prop}
\begin{proof}
In \cite{GPV} it is shown that any invariant given by an arrow
formula, where all arrow diagrams contain at most $n$ arrows,
is a GPV invariant of degree at most $n$.
A straightforward translation of this statement into the language of
state sums of Sections \ref{sub:states}-\ref{sub:conway} implies
that any invariant defined by a state sum, where the states contain
at most $n$ crossings, is of degree at most $n$.
\end{proof}

\subsection{Conway-type representation of a braid group}
If we restrict ourselves to braids instead of all $n$-tangles, we
can apply results of the previous section to obtain a following
representation of the braid group $B_n$ on $n$ strands.

Let $R_n=\mathbb{Z}S_n[z]$ be polynomials in one variable $z$ with
coefficients in the group ring ${\mathbb Z}S_n$ of the symmetric
group $S_n$. Denote by $\sigma_i$ and $s_i$, $i=1,\dots,n-1$ the
standard generators of $B_n$ and $S_n$ respectively.

\begin{prop}
For $s\in S_n$, define
$$\hat{\sigma}_i (s)=\begin{cases}s_i s & \text{if } s(i)<s(i+1)\\
                s_i s + s z & \text{if } s(i)>s(i+1)\,,
               \end{cases}
\hat{\sigma}_i^{-1} (s)=\begin{cases}s_i s & \text {if } s(i)>s(i+1)\\
                s_i s - s z & \text{if } s(i)<s(i+1)\,,
               \end{cases}
$$
Set $\hat{\sigma}_i(sz^k)=\hat{\sigma}_i(s)z^k$ and extend it to
$R_n$ by linearity. This assignment defines an action of the braid
group $B_n$ on $R_n$. This representation satisfies $\hat{\sigma}_i
(s)-\hat{\sigma}_i^{-1}(s)=z s$, thus factors through the Hecke
algebra.
\end{prop}

\begin{proof}
Use $s\in S_n$ to order inputs of all strings on the top of an
elementary tangle $\sigma_i$ by $2s_1-1,2s_2-1,\dots,2s_n-1$.
The only two orderings of the outputs of $\sigma_i$ for which
there exist a descending state are
$2s_1,\dots,2s_{i-1},2s_{i+1},2s_i,\dots,2s_n$ (for which the
trivial state is descending) and
$2s_1,\dots,2s_{i-1},2s_i,2s_{i+1},\dots,2s_n$ $s_i s$
(for which 1-state which consists of the only crossing of $\sigma_i$)
is descending. See Figure \ref{fig:sigma_i}.
This immediately leads to the expression for $\hat{\sigma}_i (s)$.
The expression for $\hat{\sigma}_i^{-1} (s)$ is obtained in the same way.
\end{proof}
\begin{figure}[htb]
\centerline{\includegraphics[width=4.4in]{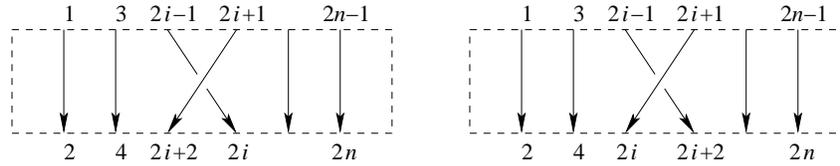}}
\caption{\label{fig:sigma_i} Two possible orderings of an elementary tangle}
\end{figure}

It would be interesting to identify this representation with
(a quotient of) a known representation of the braid group.

The above representation may be used to recover the Conway
polynomial as follows. For $s\in S_n$ and $\beta\in B_n$, denote by
$\hat{\beta}[k](s)$ the coefficient of $z^k$ in $\hat{\beta}(s)$;
this defines a linear operator on ${\mathbb Z}S_n$. Composing it
with a shift $\tau(i)=i+1$, $i=1,\dots, n-1$ , $\tau(n)=1$ and
taking the trace $\operatorname{tr}$ in ${\mathbb Z}S_n$, we define
$$c_k(\beta)=\operatorname{tr}(\tau\cdot\hat{\beta}[k])$$

\begin{prop}
Let $L$ be a link, obtained as a closure of a braid $\beta$. Then
$c_k(\beta)$ is the coefficient of $z^k$ in the Conway-Alexander
polynomial $\D(L)$ of $L$.
\end{prop}

\begin{proof}
Follows from Corollary \ref{cor:Conway}.
\end{proof}

%
%

\end{document}